\documentclass[12pt,reqno]{article}

\usepackage[usenames]{color}
\usepackage{amssymb}
\usepackage{amsmath}
\usepackage{amsthm}
\usepackage{amsfonts}
\usepackage{amscd}
\usepackage{graphicx}

\usepackage[colorlinks=true,
linkcolor=webgreen,
filecolor=webbrown,
citecolor=webgreen,
hypertexnames=false]{hyperref}

\usepackage[linesnumbered,lined,boxed,commentsnumbered]{algorithm2e}

\definecolor{webgreen}{rgb}{0,.5,0}
\definecolor{webbrown}{rgb}{.6,0,0}

\usepackage{color}
\usepackage{fullpage}
\usepackage{float}

\usepackage{graphics}
\usepackage{latexsym}
\usepackage{epsf}
\usepackage{breakurl}

\newcommand{\seqnum}[1]{\href{https://oeis.org/#1}{\rm \underline{#1}}}

\def\modd#1 #2{#1\ \mbox{\rm (mod}\ #2\mbox{\rm )}}

\begin{document}

%\begin{center}
%\epsfxsize=4in
%\leavevmode\epsffile{logo129.eps}
%\end{center}

\theoremstyle{plain}
\newtheorem{theorem}{Theorem}
\newtheorem{corollary}[theorem]{Corollary}
\newtheorem{lemma}[theorem]{Lemma}
\newtheorem{proposition}[theorem]{Proposition}

\theoremstyle{definition}
\newtheorem{definition}[theorem]{Definition}
\newtheorem{example}[theorem]{Example}
\newtheorem{conjecture}[theorem]{Conjecture}

\theoremstyle{remark}
\newtheorem{remark}[theorem]{Remark}

\begin{center}
\vskip 1cm{\LARGE\bf 
Efficient Calculation the Number of Partitions of the Set $\{1, 2, \ldots, 3n\}$ into Subsets $\{x, y, z\}$ Satisfying $x+y=z$
}
\vskip 1cm
\large
Christian Hercher\\
Institut f\"{u}r Mathematik\\
Europa-Universit\"{a}t Flensburg\\
Auf dem Campus 1b\\
24944 Flensburg\\
Germany \\
\href{mailto:christian.hercher@uni-flensburg.de}{\tt christian.hercher@uni-flensburg.de} \\
\vskip 1cm
Frank Niedermeyer\\
Bonn\\
Germany\\
\href{mailto:F.Niedermeyer@pipin11.de}{\tt F.Niedermeyer@pipin11.de}
\end{center}

\vskip .2 in

\begin{abstract}
Consider the set $\{1,2,\ldots,3n\}$. We are interested in the number of partitions of this set into subsets of three elements each, where the sum of two of them equals the third.

We give some criteria such a partition has to fulfill, which can be used for efficient pruning in the search for these partitions. In particular, we enumerate all such partitions for $n=16$ and $n=17$ adding new terms to the series \seqnum{A108235} in the Online Encyclopedia of Integer Sequences.
\end{abstract}

\section{Introduction}
The enumeration of set partitions with different constraints as pattern avoiding is a long-studied field of problems; see Mansour~\cite{Mansour} for an overview. 

We are interested in partitions of the set $\{1,\ldots,3n\}$ for a positive integer~$n$ into subsets of size~3 each with the additional property that in every such subset the sum of the two smaller elements gives the third one. 

This problem has wide connections to cyclic Steiner triple systems, modified Nim games or arranging colored blocks; see Nowakowski~\cite{Nowakowski} for an overview. 

E.g., Nickerson asks in \cite{Nickerson} for an arrangement of two copies of all integers $i\in\{1,\ldots,2n\}$ such that the two copies of $i$ are separated by exactly $i-1$ other terms. Thus, $11423243$ is a Nickerson sequence for $n=4$. If $a_1, a_2, \ldots, a_{2n}$ is such a Nickerson sequence, then let $p_1(i)$ be the index of the first occurrence of the number~$i$ and $p_2(i)$ the index of the second one. Then the sets $\{p_1(i),p_2(i)\}$ are partitioning the set $\{1,\ldots,2n\}$, and each difference $p_2(i)-p_1(i)=i$ occurs exactly once. Thus, the sets $\{i,p_1(i)+n,p_2(i)+n\}$ are a partition of $\{1,\ldots,3n\}$. Hence, every solution to Nickerson's problem leads to such a partition. (Note that the revers is not true.)

But the task of finding the number of such partitions is also a good riddle; see e.g.~\cite{MP}. That was the context in which the authors of this paper had first contact with this problem.

It is well-known that such a partition can only exist if $n\equiv \modd{0} {4}$ or $n \equiv \modd{1} {4}$. The proof of this statement is straightforward: Suppose we have a partition of $\{1,\ldots,3n\}$ into such subsets $\{x_i,y_i,z_i\}$ with $x_i+y_i=z_i$. Then, clearly, we have 
\[\sum_i (x_i+y_i)=\sum_i z_i=\frac{1}{2} \sum_{k=1}^{3n} k=\frac{3n(3n+1)}{4}.\]
So $4\mid 3n(3n+1)$. Since $3n$ and $3n-1$ are relatively prime it follows $4\mid 3n$ or $4\mid 3n+1$, and hence, we get $n\equiv \modd{0} {4}$ or $n\equiv \modd{1} {4}$, $\Box$.

If a positive integer $n$ is of the necessary form, then such partitions of $\{1,\ldots,3n\}$ always exist. Examples of general constructions are given by Nowakowski~\cite{Nowakowski}.

The question of how many such partitions for a given $n$ exist is not so easy to answer. In the next sections, we give tools for an efficient computation of these numbers.

\section{Key Ideas}
In the following, we want to use a recursive approach to search for such partitions: Find one subset $\{x,y,z\}$ with $x+y=z$, eliminate those numbers and search for the next one in the remaining set until all numbers are partitioned. This can be done as in the naive recursive Algorithm~\ref{Algo_0}.

\begin{algorithm}[H]
 \SetAlgoLined
 \SetKwProg{Fn}{Function}{:}{end}
  \KwIn{A nonempty set $S=\{b_1,b_2,\ldots,b_{3m}\}\subseteq\{1,\ldots,3n\}$ with $b_1<b_2<\cdots<b_{3m}$}
  \KwOut{The number of partitions $P=\{\{x_1,y_1,z_1\},\ldots,\{x_m,y_m,z_m\}\}$ of $S$ into pairwise disjoint subsets $\{x_i,y_i,z_i\}$ satisfying $x_i+y_i=z_i$ for all $1\leq i\leq m$}
  \Fn{Counting\_Partitions($S$)}
  {
    \uIf(\tcp*[f]{End of Recursion}){$m=1$}
  {
    \uIf{$b_1+b_2=b_3$}{\Return{1}}\uElse{\Return{0}}
  }
  \BlankLine
  $number\_of\_partitions \leftarrow 0$\;
  \BlankLine
    \For{$i=2$ \KwTo $3m-1$}
    {
      \uIf(\tcp*[f]{found a triple $(x,y,z)=(b_1,b_i,b_k)$}){there is a $i<k\leq 3m$ with $b_1+b_i=b_k$}
      {
      $S^{\prime}\leftarrow S\setminus \{b_1,b_i,b_k\}$\;
      $number\_of\_partitions += Counting\_Partitions(S^{\prime})$\;
      }
    }
    \Return{$number\_of\_partitions$}\;
  }
\caption{Naive Recursive Algorithm.}
\label{Algo_0}
\end{algorithm}

For a much faster calculation, we want to know as early as possible if a complete such partition is not possible anymore. The first and most used constraint is given in the next theorem.

\begin{theorem}\label{Bedingung_1}
Let $m$ be a positive integer, and $S:=\{b_1,b_2,\ldots,b_{3m}\}$ a set of positive integers with $b_1<b_2<\cdots<b_{3m}$ which can be partitioned into subsets $\{x_i,y_i,z_i\}$, $i=1, \ldots, m$ with $x_i+y_i=z_i$. Furthermore, let $S_1:=\sum_{k=1}^{2m} b_k$ and $S_2:=\sum_{k=2m+1}^{3m} b_k$. Then
\[S_1\leq S_2.\]
\end{theorem}

\begin{proof}
 For a given set $S=\{b_1,b_2,\ldots,b_{3m}\}$ with  $b_1<b_2<\cdots<b_{3m}$ let $\bigcup_{i=1}^m \{x_i,y_i,z_i\}=S$ and $x_i+y_i=z_i$ for all $i=1,\ldots,m$. Then we have
  \[S_1 = \sum_{k=1}^{2m} b_k \leq \sum_{i=1}^m (x_i+y_i) = \sum_{i=1}^m z_i \leq \sum_{k=2m+1}^{3m} b_k = S_2.\] 
\end{proof}

This leads to the following basic Algorithm~\ref{Algo_1} for enumerating all such partitions of subsets of $\{1,\ldots,3n\}$.

\begin{algorithm}[H]
 \SetAlgoLined
 \SetKwProg{Fn}{Function}{:}{end}
  \KwIn{A nonempty set $S=\{b_1,b_2,\ldots,b_{3m}\}\subseteq\{1,\ldots,3n\}$ with $b_1<b_2<\cdots<b_{3m}$, the sums $S_1:=\sum_{i=1}^{2m} b_i$ and $S_2:=\sum_{i=2m+1}^{3m} b_i$}
  \KwOut{The number of partitions $P=\{\{x_1,y_1,z_1\},\ldots,\{x_m,y_m,z_m\}\}$ of $S$ into pairwise disjoint subsets $\{x_i,y_i,z_i\}$ satisfying $x_i+y_i=z_i$ for all $1\leq i\leq m$}
  \Fn{Counting\_Partitions($S$, $S_1$, $S_2$)}
  {
    \uIf(\tcp*[f]{End of Recursion}){$m=1$}
  {
    \uIf{$S_1=S_2$}{\Return{1}}\uElse{\Return{0}}
  }
  \BlankLine
  \uIf(\tcp*[f]{Condition of Theorem~\ref{Bedingung_1} does not hold}){$S_1>S_2$}{\Return{0}}
  \BlankLine
  $number\_of\_partitions \leftarrow 0$\;
  \BlankLine
    \For{$i=2$ \KwTo $3m-1$}
    {
      \uIf(\tcp*[f]{found a triple $(x,y,z)=(b_1,b_i,b_k)$}){there is a $i<k\leq 3m$ with $b_1+b_i=b_k$}
      {
      $S^{\prime}\leftarrow S\setminus \{b_1,b_i,b_k\}$\;
      \uIf{$i\geq 2m+1$}
      {
        $S_1^{\prime}\leftarrow S_1-b_1-b_{2m}$\;
        $S_2^{\prime}\leftarrow S_2-b_i-b_k+b_{2m}$\;
      }
      \uElseIf{$k\geq 2m+1$}
      {
        $S_1^{\prime}\leftarrow S_1-b_1-b_i$\;
        $S_2^{\prime}\leftarrow S_2-b_k$\;
      }
      \uElse
      {
        $S_1^{\prime}\leftarrow S_1-b_1-b_i-b_k+b_{2m+1}$\;
        $S_2^{\prime}\leftarrow S_2-b_{2m+1}$\;
      }
      \BlankLine
      
      $number\_of\_partitions += Counting\_Partitions(S^{\prime},S_1^{\prime},S_2^{\prime})$\;
      }
    }
    \Return{$number\_of\_partitions$}\;
  }
\caption{Basic Algorithm.}
\label{Algo_1}
\end{algorithm}

Furthermore, we have some extra conditions such a partition must have:

\begin{theorem}\label{Thm_2}
In the same setting as in Theorem~\ref{Bedingung_1} the following conditions hold:
\begin{enumerate}
 \item \label{Bedingung_2} If and only if $S_1=S_2$, then $\{x_i \mid 1\leq i\leq n\} \cup \{y_i \mid 1\leq i\leq n\} = \{b_1, b_2, \ldots, b_{2n}\}$ and $\{z_i \mid 1\leq i\leq n\} = \{b_{2n+1}, b_{2n+2}, \ldots, b_{3n}\}$.
 \item \label{Bedingung_3} It is $b_1+b_{2n}\leq b_{3n}$.
\item \label{Bedingung_4} If equality holds in \ref{Bedingung_3}, then $\{x_i, y_i, z_i\}=\{b_1, b_{2n}, b_{3n}\}$ for some $i$ and $S_1=S_2$.
\item \label{Bedingung_5} If $b_1+b_{2n+1}>b_{3n}$, then $S_1=S_2$.  
\end{enumerate}
\end{theorem}

\begin{proof}
We use the same notation as in the proof of Theorem~\ref{Bedingung_1}. 
\begin{enumerate}
\item If equality holds in Theorem~\ref{Bedingung_1}, then it must hold on all places in the chain of inequalities given in the last item. If $S_1 = \sum_{i=1}^n (x_i+y_i)$, then $\{x_i, y_i \mid 1\leq i\leq n\} = \{b_1, b_2, \ldots, b_{2n}\}$. And if $\sum_{i=1}^n z_i = S_2$, we have $\{z_i \mid 1\leq i\leq n\} = \{b_{2n+1}, b_{2n+2}, \ldots, b_{3n}\}$. And if the sets are as described, equality holds in Theorem~\ref{Bedingung_1}.
\item Clearly, we have $x_i\geq b_{2n}$ or $y_i\geq b_{2n}$ for some $i$. W.l.o.g.\ let the second case be true. Then we have $b_1+b_{2n}\leq x_i+y_i=z_i\leq b_{3n}$.
\item If equality holds in Theorem~\ref{Bedingung_1}, then from the chain of inequalities in the previous item we get $x_i=b_1$, $y_i=b_{2n}$, and $z_i=b_{3n}$, which proves the first part of this claim. Furthermore, if there were an index $j$ with $x_j<y_j$ and $y_j>b_{2n}$, we would have $b_1+b_{2n}<x_j+y_j=z_j\leq b_{3n}$, contradicting the assumed equality in Theorem~\ref{Bedingung_1}. Therefore, all $x_i$ and $y_i$ are smaller than or equal to $b_{2n}$, which proves the second part of this claim.
\item Let $j\in\{1,2,\ldots,n\}$, and w.l.o.g.\ let $x_j<y_j$. If $b_1+b_{2n+1}>b_{3n}$, we have $b_{3n}\geq z_j=x_j+y_j\geq b_1+y_j > b_{3n}-b_{2n+1}+y_j$. Thus, it is $b_{2n+1}>y_j>x_j$, and we get equality in Theorem~\ref{Bedingung_1} with the second direction of the equivalence given in Condition~\ref{Bedingung_2}.
\end{enumerate} 
\end{proof}

This leads to some significant improvements of Algorithm~\ref{Algo_1}:

\begin{itemize}
 \item If at some point in the recursive search the equality $S_1=S_2$ holds, then any triple has to consist of two elements of $\{b_1,\ldots,b_{2m}\}$ and one of $\{b_{2m+1},\ldots,b_{3m}\}$. Thus, we have $2\leq i\leq 2m$ and $2m+1\leq k\leq 3m$, which narrows the search space for triples. Also, if this equality is reached at one point, it has to hold in all subsequent recursion steps. Therefore, in this case, the sums $S_1$ and $S_2$ need not be computed anymore. This will be done in a reduced version of the given function.
 \item If $b_1+b_i\geq b_k$, then $b_1+b_{i+1}>b_k$. Therefore, after the search for triples with $b_1$ and $b_i$ up to $b_k$, the search for triples with $b_1$ and $b_{i+1}$ can start with the next value after $b_k$.
 \item If for some $i$ one has $b_1+b_i>b_{3m}$, then one can end the search not only for this $i$ but also for all higher values of $i$ in the \texttt{for} loop. 
 \item If $b_1+b_i>b_{3m}$ for some $i\geq 2m+1$, then by Condition~\ref{Bedingung_5} of Theorem~\ref{Thm_2}, we can use the reduced function for further calculating, since we have $S_1=S_2$.
 \item If $b_1+b_{2m}=b_{3m}$, we have found one triple $\{b_1,b_{2m},b_{3m}\}$, and the next step of the recursion can be done with the reduced function, since we have $S_1=S_2$ in this case as well.
\end{itemize}

With this we improved the program by a fair bit; see Tables~\ref{Table_1} and \ref{Table_2} for comparison. Its source code is given in an additional document. 

With this we calculated the up to then unknown values for $n=16$ and $n=17$ listed in the series \seqnum{A108235} of the Online Encyclopedia of Integer Sequences. We get

\begin{align*}
a_{16} &= 142664107305\text{ and}\\
a_{17} &= 1836652173363.
\end{align*}

\section{Further Ideas}
Theorems~\ref{Bedingung_1} and \ref{Thm_2} as well as the algorithms we provided do not make use of the fact that the set $S$ to be partitioned is a subset of $\{1,\ldots,3n\}$. It could be any set of positive reals. Thus, we can use them to answer some slight variation of the given problem, too: 

Sedl\'{a}\v{c}ek~\cite{Sedlacek} asks if, for every $n\not\equiv \modd{6} {12}$ and $n\not\equiv \modd{9} {12}$ in the set $\{1,\ldots,n\}$, there are $\lfloor\tfrac{n}{3}\rfloor$ pairwise disjoint subset $\{x,y,z\}$ with $x+y=z$. (For $n\equiv \modd{6} {12}$ or $n\equiv \modd{9} {12}$ the parity argument made in the introduction proves that there cannot be $\tfrac{n}{3}$ such triples.) Guy~\cite{Guy} gives a general proof of existence for all remaining $n$ and asks how many different such maximal collections
of disjoint subsets with the desired property exist. The sequence \seqnum{A002849} in OEIS gives these numbers. (For $n\equiv \modd{6} {12}$ or $n\equiv \modd{9} {12}$ the maximal number of disjoint subsets only is $\tfrac{n}{3}-1$.)

By excluding the right quantity of numbers from $\{1,\ldots, n\}$ such that the remaining set has a multiple of~3 elements and their sum is divisible by~4, we construct subproblems which can be solved with our methods. In this way, we calculated the next hitherto unknown elements in this series:

\begin{align*}
 a_{43}&=16852166906 \text{ and}\\
 a_{44}&=270947059160.
\end{align*}

But, if we constrain our search to the task of partitioning the set $\{1,\ldots,3n\}$ (with $n\equiv \modd{0} {4}$ or $n\equiv \modd{1} {4}$ as given in the introduction), then we can use this as a meta information. Thus, we know that all elements $b_1,\ldots,b_{3m}$ of the considered set are positive integers and at most $3n$. Additionally, by our strategy of always putting the smallest available element in the next triple, we know $b_1\geq n-m+1$, since, previously, at least the smallest $n-m$ elements have already been chosen for the first $n-m$ triples. 

\begin{theorem}\label{Thm_3}
In the recursive search for partitions of $\{1,\ldots,3n\}$ into subsets $\{x,y,z\}$ satisfying $x+y=z$, where in every step the smallest available number is chosen, suppose that only be six elements remain. Thus, we have $m=2$. Let $S=\{b_1,\ldots,b_6\}$ with $b_1<\cdots<b_6$, $S_1=b_1+\cdots+b_4$, and $S_2=b_5+b_6$. If there is such a partition, then $S_1=S_2$.   
\end{theorem}

\begin{proof}
Since $m=2$, we look at the second to last step in the recursion: Only six elements $b_1<\cdots<b_6$ are remaining; and we know $b_1+b_2+b_3+b_4\leq b_5+b_6$. From the previous paragraph, we additionally have $n-1\leq b_1$, $n\leq b_2$, $\ldots$ and $b_6\leq 3n$.

Now we can make an analysis by case which partitions are still possible. In the following, we distinguish the cases by naming the ``partners'' of $b_1$ in its triple. The other triple is then well-defined:

\begin{itemize}
 \item $b_1+b_2=b_3$: Then $b_3\geq (n-1)+n=2n-1$ and $b_4\geq 2n$. But this leads to the contradiction $3n\geq b_6=b_5+b_4 > 2n+2n=4n$.
 \item $b_1+b_2=b_4$: Then $b_4\geq 2n-1$, and $b_5\geq 2n$. This yields the contradiction $n+1\leq b_3\leq b_6-b_5\leq 3n-2n=n$.
 \item $b_1+b_2=b_5$: A possible solution.
 \item $b_1+b_2=b_6$: This cannot be the case, since $b_1+b_2<b_3+b_4=b_5<b_6$.
 \item $b_1+b_3=b_4$: Then $b_4\geq (n-1)+(n+1)=2n$ and $b_5\geq 2n+1$. Because of $n+1\leq b_2=b_6-b_5\leq 3n-(2n+1)=n-1$, this leads to a contradiction, too.
 \item $b_1+b_3=b_5$: A possible solution.
 \item $b_1+b_3=b_6$: This cannot be the case, since $b_1+b_3<b_2+b_4=b_5<b_6$.
 \item $b_1+b_4=b_5$: A possible solution.
 \item $b_1+b_4=b_6$: A possible solution.
 \item $b_1+b_5=b_6$: Then we have $b_4=b_2+b_3\geq n+(n+1)=2n+1$ and, therefore, $b_5\geq 2n+2$. This gives the contradiction $3n\geq b_6=b_1+b_5\geq (n-1)+(2n+2)=3n+1$.
\end{itemize}

In all four remaining cases, we have $b_1+b_2+b_3+b_4=b_5+b_6$, hence $S_1=S_2$.
\end{proof}

If this equality does not hold, we can discard the set directly. Otherwise, we only need to check the four possible solutions from above. This speeds up the process, a little bit further. 

A comparison of running times and number of recursive function calls of the different algorithms was done on a i9 processor of the 11th generation. The algorithms were not parallelized. So only one core was used. The different running times are given in Table~\ref{Table_1} and the number of recursive function calls in Table~\ref{Table_2}.

\begin{table}[htb]
\begin{center}
\begin{tabular}{l|r|r|r|r|r}
 & \multicolumn{5}{c}{Running Time for}\\
 & $n=8$ & $n=9$ & $n=12$ & $n=13$ & $n=16$\\
\hline
Number of Solutions & 3040 & 20505 & 10567748 & $103372655$ & $\approx 1.43 \cdot 10^{11}$\\
\hline\hline
Naive Algorithm~\ref{Algo_0} & 16 ms & 146 ms & 308 s & $>$24 h & --- \\
Basic Algorithm~\ref{Algo_1} & 2 ms & 18 ms & 19 s & 219 s & --- \\
Algorithm with Theorem~\ref{Thm_2} & 1 ms & 11 ms & 10 s & 110 s & 211658 s\\
Algorithm with Theorems~\ref{Thm_2} \& \ref{Thm_3} & 1 ms & 11 ms & 10 s & 107 s & 199610 s
\end{tabular}
\caption{Running Times of the different Algorithms.}
\label{Table_1}
\end{center}
\end{table}

\begin{table}[htb]
\begin{center}
\begin{tabular}{l|r|r|r|r|r}
 & \multicolumn{5}{c}{Number of recursive function calls for}\\
 & $n=8$ & $n=9$ & $n=12$ & $n=13$ & $n=16$\\
\hline
Naive Algorithm~\ref{Algo_0} & 435083 & 4567652 & $88.00\cdot 10^8$ & --- & --- \\
Basic Algorithm~\ref{Algo_1} & 49059 & 401092 & $3.77 \cdot 10^8$ & $4.36\cdot 10^9$ & --- \\
Algorithm with Theorem~\ref{Thm_2} & 39793 & 307826 & $2.49 \cdot 10^8$ & $2.73 \cdot 10^9$ & $5.432 \cdot 10^{12}$\\
Algorithm with Theorems~\ref{Thm_2} \& \ref{Thm_3} & 36103 & 287085 & $2.43 \cdot 10^8$ & $2.69 \cdot 10^9$ & $5.405 \cdot 10^{12}$
\end{tabular}
\caption{Number of recursive function calls of the different Algorithms.}
\label{Table_2}
\end{center}
\end{table}

\bigskip
\hrule
\bigskip

\noindent 
2020 \emph{Mathematics Subject Classification}:~Primary 05A18

\medskip

\noindent 
\emph{Keywords}: set partitions with arithmetic property 
\end{document}